\documentclass[12pt]{amsart}
\usepackage[colorlinks=true,pagebackref,hyperindex,citecolor=black,linkcolor=black]{hyperref}
\usepackage[top=1in, bottom=1in, left=1.	1in, right=1.1in]{geometry}
\usepackage{amsmath}
\usepackage{amsfonts}
\usepackage{amssymb}
\usepackage{amsthm}
\usepackage{stmaryrd}
\usepackage[all]{xy}  
\usepackage{mathrsfs}
\usepackage{enumitem}
\usepackage[T1]{fontenc}

\usepackage{color}

\makeatletter
\def\namedlabel#1#2{\begingroup
#2%
\def\@currentlabel{#2}%
\phantomsection\label{#1}\endgroup
}
\makeatother

\newtheorem{theoremx}{Theorem}
\theoremstyle{theorem}
\newtheorem{theorem}{Theorem}[section]

\newtheorem{corollary}[theorem]{Corollary}
\newtheorem{lemma}[theorem]{Lemma}
\newtheorem{proposition}[theorem]{Proposition}

\theoremstyle{definition}
\newtheorem{definition}[theorem]{Definition}
\newtheorem{question}[theorem]{Question}
\newtheorem{example}[theorem]{Example}
\newtheorem{remark}[theorem]{Remark}
\numberwithin{equation}{subsection}

\newcommand{\NN}{\mathbb{N}}
\newcommand{\RR}{\mathbb{R}}
\newcommand{\ZZ}{\mathbb{Z}}
\newcommand{\QQ}{\mathbb{Q}}

\newcommand{\m}{\mathfrak{m}}
\newcommand{\n}{\mathfrak{n}}

\newcommand{\cA}{\mathscr{A}}
\newcommand{\cB}{\mathscr{B}}

\newcommand{\cR}{\mathscr{R}}

\newcommand{\Spec}{\operatorname{Spec}}
\newcommand{\Hom}{\operatorname{Hom}}

\newcommand{\reg}{\operatorname{reg}}	
\newcommand{\fpt}{\operatorname{fpt}}	
\newcommand{\gr}{\operatorname{gr}}	
\newcommand{\e}{\operatorname{e}}	
             
\newcommand{\ds}{\displaystyle}

\renewcommand{\a}{\mathfrak{a}}

\usepackage{soul}

\definecolor{blue-violet}{rgb}{0.54, 0.17, 0.89}
\definecolor{Blue}{rgb}{0.01, 0.28, 1.0}
\definecolor{gGreen}{rgb}{0.2, 0.8, 0.2}
\definecolor{Green}{rgb}{0.04, 0.85, 0.32}

\begin{document}

\title{On the existence of $F$-thresholds and related limits}
\dedicatory{Dedicated to Professor~Craig~Huneke on the occasion of his sixty-fifth birthday.}
\author{Alessandro De Stefani}
\address{Department of Mathematics, Royal Institute of Technology (KTH), 100 44 Stockholm, Sweden}
\email{ads@kth.se}
\author{Luis N\'u\~nez-Betancourt$^*$}
\address{Centro de Investigaci\'on en Matem\'aticas, Guanajuato, GTO, M\'exico}  \email{luisnub@cimat.mx}
\thanks{$^*$ The second author was partially supported by NSF Grant 1502282.}
\author{Felipe P\'erez}
\address{Department of Mathematics \& Statistics, Georgia State University, Atlanta, GA 30303, USA}  
\email{jperezvallejo@gsu.edu}
\subjclass[2010]{Primary 13A35; Secondary 13D45, 14B05.}
\keywords{$F$-thresholds, $a$-invariants, $F$-pure thresholds, Castelnuovo-Mumford regularity; strong $F$-regularity; $F$-signature.}

\maketitle

\begin{abstract}

We show the existence of $F$-thresholds in full generality.  In addition, we study properties of standard graded algebras over a field for which $F$-pure threshold and $F$-threshold at the irrelevant maximal ideal agree. We also exhibit explicit bounds for the $a$-invariants and Castelnuovo-Mumford regularity of Frobenius powers of ideals in terms of $F$-thresholds and $F$-pure thresholds, obtaining the existence of related limits in certain cases.

\end{abstract}

\section{Introduction}

In recent years there has been an intense research in algebraic geometry and commutative algebra, carried towards a better understanding of what are nowadays known as $F$-singularities. Particular attention has been given to $F$-pure \cite{HRFpurity,FedderFpurityFsing} and $F$-regular singularities \cite{HoHu1,HoHu2,HoHu3,FW}. Attached to these singularity types, there are numerical invariants that measure how good or bad a singular point is; for instance, the $F$-thresholds \cite{MTW,HMTW}, the $F$-pure thresholds \cite{TW2004}, and the $F$-signature \cite{SmithVDB,HLMCM,TuckerFSig}.
In this manuscript, we study these numbers, compare them, and obtain consequences regarding the singularities of the ring.

The $F$-thresholds were first introduced for regular rings by Musta{\c{t}}{\u{a}}, Takagi, and Watanabe \cite{MTW}, as a positive characteristic analogue of log-canonical thresholds. In a subsequent joint work with Huneke \cite{HMTW}, $F$-thresholds were defined in general rings of prime characteristic as limits of normalized Frobenius orders,  provided they exist. In the same article, the authors showed compelling relations that $F$-thresholds have with the Hilbert-Samuel multiplicity, tight closure and integral closure. However, a drawback of using these methods was that the convergence of the sequence defining the $F$-thresholds had been shown only in some partial cases \cite{HMTW,HTW,LiSocles}. We settle this problem by proving the existence of $F$-thresholds in full generality.

\begin{theoremx}[{see Theorem \ref{Existence c}}] Let $R$ be a Noetherian ring of prime characteristic $p$. Let $\a, J\subseteq R$ be ideals such that $\a \subseteq \sqrt{J}$. If $\nu_\a^J(p^e) := \max\{t \in \NN \mid \a^t \not\subseteq J^{[p^e]} \}$, then the $F$-threshold $c^J(\a) = \lim\limits_{e\to \infty} \frac{ \nu^J_\a(p^{e})}{p^e}$ exists.
\end{theoremx}

In the rest of this article, we study relations between $F$-thresholds and other invariants in prime characteristic. Namely, $F$-pure thresholds,  $a$-invariants, and asymptotic Castelnuovo-Mumford regularity of Frobenius powers.

The $F$-pure threshold of an ideal $\a \subseteq R$, denoted $\fpt(\a)$, was defined by Takagi and Watanabe \cite{TW2004}. Roughly speaking, the $F$-pure threshold of an ideal measures its splitting order. The general expectation is that, the higher the $F$-pure threshold, the better the singularities \cite{H-Y,BFS}. Takagi and Watanabe \cite{TW2004} write: {\it ``Although our first motivation was to investigate the log canonical threshold via the F-pure threshold, we find that the F-pure threshold itself is an interesting invariant in commutative algebra''}. 

If $(R,\m,K)$ is regular, either local or standard graded, then $\fpt(\a) = c^\m(\a)$ for any ideal $\a \subseteq R$. In contrast,  this is often not the case for singular rings. However, the inequality $\fpt(\a) \leq c^\m(\a)$ holds true in general. In this article, we focus on the study of the following question, asked in different settings by several researchers.

\begin{question}[{\cite{Hirose,DiagF-thresholdsBinHyp,HWY}}]\label{Question Equality}
What are necessary and sufficient conditions for the equality $\fpt(\m)=c^\m(\m)$ to hold? 
\end{question}

Our first step towards an answer to Question \ref{Question Equality} is a characterization of the $F$-pure threshold as a limit of $F$-thresholds (see Theorem \ref{Main 1} and Corollary \ref{Cor fpt=c}). If we restrict ourselves to standard graded Gorenstein $K$-algebras, we are able to partially answer Question \ref{Question Equality}, giving a necessary condition for the equality to hold. In a sense, this result says that standard graded Gorenstein rings such that $\fpt(\m) = c^\m(\m)$ have the best possible type of $F$-singularities.

\begin{theoremx} (see Theorem \ref{main thm})\label{Main GorGraded2}
Let $(R,\m,K)$ be a $d$-dimensional standard graded Gorenstein $K$-algebra that is $F$-finite and $F$-pure. If $\fpt(\m)=c^\m(\m),$ then $R$ is strongly $F$-regular.
Furthermore,\[
s(R) \geq \frac{e(R)}{d!},
\]
where $e(R)$ denotes the Hilbert-Samuel multiplicity and $s(R)$ the $F$-signature of $R$.
\end{theoremx}

Using recent results of Singh, Takagi, and Varbaro \cite{STV}, we can extend Theorem \ref{Main GorGraded2} to normal standard graded Cohen-Macaulay algebras whose anti-canonical cover is Noetherian (see Corollary \ref{coroll anticancover}). These include $\QQ$-Gorenstein algebras. 

The inequality for the $F$-signature in the previous theorem is particularly meaningful because lower bounds for this invariant are typically hard to produce. The first two authors prove that $\fpt(\m)\leq -a_{\dim(R)}(R)\leq c^{\m}(\m)$ \cite{DSNBFpurity}. To obtain Theorem \ref{Main GorGraded2}, we need to extend 
these relations to the  Castelnuovo-Mumford regularity and $a$-invariants of 
Frobenius powers. Recall that, for a finitely generated $R$-module $M$, the Castelnuovo-Mumford regularity of $M$ can be defined as $\reg(M) = \max\{a_i(M)+i \mid i \in \NN\}$, where $a_i(M) = \sup\{s \in \ZZ \mid [H^i_\m(M)]_s \ne 0\}$, with the convention that $\sup(\emptyset) = -\infty$. We point out that the growth of $a_i(R/J^{[p^e]})$ and $\reg(R/J^{[p^e]})$ are of independent interest, since they are connected to discreteness of $F$-
jumping coefficients \cite{KatzmanZhang,KSSZ,ZhangRegFrob}, localization of tight closure \cite{KatzmanComplexityFrob,HunekeLC}, and existence of the generalized Hilbert-Kunz multiplicity \cite{DaoSmirnov, VraciugHK}. The following theorem is the main result we prove in this direction, and it is a key ingredient in the proof of Theorem \ref{Main GorGraded2}. 
 
\begin{theoremx}[{see Theorems \ref{Thm Limit reg-inv}, \ref{ThmRegDimR/J_t} and \ref{ThmRegDimR/J_t-1}}] \label{Main Lim Regularity}
Let $(R,\m,K)$ be a standard graded $K$-algebra that is $F$-finite and $F$-pure. Suppose that $J\subseteq R$ is a homogeneous ideal.
If  there exists a constant $C$ such that $\reg(R/J^{[p^e]})\leq Cp^e$ for all $e \gg 0$, then $\lim\limits_{e\to \infty} \frac{\reg(R/J^{[p^e]})}{p^e}$ exists, and it is bounded below by $\max_{i \in \NN}\{a_i(R/J)\}+\fpt(\m)$.

Furthermore,  for $t=\dim(R/J)$ and $D=\max\{t \in \NN \;|\;\left[ J/\m J\right]_t \neq 0\} +1$. If $H^{t-1}_\m(R/J^{[p^e]}) \ne 0$ for some $e \in \NN$, then
\[
\lim\limits_{e\to \infty} \frac{a_t(R/J^{[p^e]})}{p^e}\leq a_t(R/J)+c^{\m}(\m) \ \hbox{  and } \ \lim\limits_{e\to \infty} \frac{a_{t-1}(R/J^{[p^e]})}{p^e}\leq D(c^J(J)+1).
\]
In particular, the two limits exist.
\end{theoremx}

As a consequence of Theorem \ref{Main Lim Regularity}, we obtain explicit formulas for the top $a$-invariant of the ring modulo Frobenius powers of an ideal, under the assumption that $\fpt(\m) = c^\m(\m)$ (see Proposition \ref{Prop equi frob a-inv}).

\section{Notation and preliminaries}

Throughout this article, $R$ denotes a Noetherian commutative ring with identity. We write $(R,\m,K)$ to mean either a local ring or a standard graded $K$-algebra. A standard graded algebra is a positively graded algebra over a field $K$, generated by finitely many elements of degree one. The ideal generated by the positive degree elements, that we denote by $\m$, is called the irrelevant maximal ideal. We denote by $\mu(M)$ the minimal number of generators of an $R$-module $M$, homogeneous in the graded case. We use $\lambda(M)$ to denote its length as an $R$-module. 
We make the convention that $0\in\NN$.

When the characteristic of $R$ is a positive prime integer $p$, we can consider the Frobenius endomorphism $F:R \to R$, that raises any element of $R$ to its $p$-th power. In this way, $R$ can be viewed as an $R$-module by restriction of scalars via $F$, and we denote this module action on $R$ by $F_* R$. The action is explicitly given as follows: for $r \in R$ and $F_*x \in F_* R$, we have $r \cdot F_*x = F_*(r^p x) \in F_* R$. For an integer $e \geq 1$, we can also reiterate the map $F$, and obtain a ring endomorphism $F^e:R \to R$ which is such that $F^e(r) = r^{p^e}$ for all $r \in R$. For any $R$-module $M$, we can consider the $R$-module $F^e_*M$, whose action is induced by restriction of scalars via $F^e$, as illustrated above in the case $e=1$ and $M=R$. For an ideal $J\subseteq R$, we denote by $J^{[p^e]}$ the ideal generated by $F^e(J)$, that is, the ideal generated by the $p^e$-th powers of elements in $J$. We note that $J^{[p^0]}=J$.

If $R$ is reduced, then, for all integers $e \geq 1$, the map $F^e$ can be identified with the $R$-module inclusion $R \subseteq R^{1/p^e}$, where $R^{1/p^e}$ denotes the ring of $p^e$-th roots of elements in $R$. This viewpoint can be  helpful to keep in mind, but it is not exploited further in this article.

\begin{definition} The ring $R$ is called $F$-finite if  $F_*R$ is a finitely generated $R$-module.
\end{definition}

Equivalently, $R$ is $F$-finite if $F^e_*R$ is a finitely generated $R$-module for some (equivalently, for all) integer $e \geq 1$. 

\begin{remark} If $(R,\m,K)$ is local, then $R$ is $F$-finite if and only if it is excellent, and $[F_*K:K]<\infty$ \cite[Corollary 2.6]{F-finExc}. If $(R,\m,K)$ is standard graded, then $R$ is $F$-finite if and only if $[F_*K:K]<\infty$ \cite[Lemma 1.5]{FedderFpurityFsing}.
\end{remark}

The  notion of $F$-purity was introduced by Hochster and Roberts \cite{HRFpurity}. Since then, it has played a very crucial role in the theory of singularities of rings of positive characteristic.

\begin{definition}[\cite{HRFpurity}] Let $R$ be a Noetherian ring of prime characteristic, and let $F:R \to R$ be the Frobenius endomorphism. Then $R$ is called {\it $F$-pure} if $F$ is a pure morphism, that is, $F \otimes 1: R\otimes M \to R \otimes M$ is injective for all $R$-modules $M$. The ring $R$ is called {\it $F$-split} if $F$ is a split monomorphism.
\end{definition}

If $R$ is $F$-finite, then $R$ is $F$-split if and only if $R$ is $F$-pure  \cite[Corollary 5.3 \& Proposition 5.5]{HRFpurity}. More explicitly, when $R$ is $F$-finite, we have that $R$ is $F$-pure if and only if it is reduced, and the natural inclusion $R \subseteq F^e_*R$ of $R$-modules splits for some (equivalently, for all) $e \geq 1$. 

We now recall  the graded version of Fedder's Criterion, which characterizes $F$-pure rings that are quotients of regular rings. This  result is needed to establish some reductions for Theorem \ref{main thm}.

\begin{theorem}[{\cite[Theorem 1.12]{FedderFpurityFsing}}]\label{Fedder}
Let $S=K[x_1,\ldots,x_n]$ be a polynomial ring over a field of prime characteristic $p$. Suppose that $\deg(x_i)=1,$ and $\n=(x_1,\ldots,x_n)$.
Let $I\subseteq \n$ be a homogeneous ideal. Under these assumptions $S/I$ is $F$-pure if and only if 
$(I^{[p]}:_SI)\not\subseteq \n^{[p]}.$
\end{theorem}

\section{The $F$-threshold of $\a$ with respect to $J$}

The $F$-thresholds are invariants of rings in positive characteristic obtained by comparing powers of an ideal $\a$ with Frobenius powers of another ideal $J$. They were first introduced in the regular ring setting \cite{MTW} and, later, generalized to a wider class of rings \cite{HMTW}. The $F$-thresholds were originally defined as limits of sequences of rational numbers, whenever such sequences were convergent. However, their existence remained an open problem. In this section, we show that $F$-thresholds exist in general.

\begin{definition}
Let $R$ be a ring of prime characteristic $p$. For $\a,J$ two ideals  of  $R$  satisfying $\a \subseteq \sqrt{J}$, and a non-negative integer $e$, we define
\[ 
\nu^J_\a(p^e):=\max\{t \in \NN \mid \a^t \not\subseteq J^{[p^e]} \}.
\] 
\end{definition}

The following lemma  is  well-known. We include the proof for the sake of completeness.

\begin{lemma}\label{Lemma Obs pe u(I)}
Let  $R$ be a ring of prime characteristic $p$, and $\a$ be an ideal. Then, for every $s,e\in\NN$ and  $r\geq (\mu(\a)+s-1) p^e$, we have that  $\a^r= \a^{r-s p^e}\left(\a^{[p^e]}\right)^{s}$.
\end{lemma}
\begin{proof}
Let $u=\mu(\a),$ and let $f_1,\ldots, f_u$ denote a minimal set of generators for 
$I.$
We proceed by induction on $s$. For  $s=0$, the statement is clear as  $ \a^{r-0 p^e}\left(\a^{[p^e]}\right)^0=\a^r R=\a^r$.

We now assume that our claim is true for $s$, and prove it for $s+1.$
Suppose that $r\geq  (u+s)p^e$. Then
\begin{align*}
\a^r &= \a^{r-s p^e}\left(I^{[p^e]}\right)^{s}  \hbox{ by induction hypothesis since } r\geq  (u+s-1)p^e\\
&= \a^{r-(s+1) p^e}\a^{[p^e]}\left(\a^{[p^e]}\right)^{s}
 \hbox{ by the case }s=1\hbox{ because } r-s p^e\geq  up^e\\
&= \a^{r-(s+1) p^e}\left(\a^{[p^e]}\right)^{s+1}. 
\end{align*}
\end{proof}

We point out that the following lemma has been previously stated,  without proof, for reduced rings \cite[Remark 1.5]{HTW}.

\begin{lemma}\label{Bound of Ve}
Let  $R$ be a ring of prime characteristic $p$. Let 
$\a, J\subseteq R$ be ideals such that $\a\subseteq \sqrt{J}$. Then,
 \[ \frac{\nu^J_\a(p^{e_1+e_2})}{p^{e_1+e_2}}-
\frac{\nu^J_\a(p^{e_1})}{p^{e_1}}\leq \frac{\mu(\a)}{p^{e_1}}
\]
for every $e_1,e_2\in\NN.$
\end{lemma}
\begin{proof}
Taking $s= \nu^J_\a(p^{e_1})+1$ in Lemma  \ref{Lemma Obs pe u(I)}  yields
\begin{align*}
\a^{p^{e_2} \mu(\a) +p^{e_2} \nu^J_\a(p^{e_1})} &\subseteq 
\a^{p^{e_2} \mu(\a)-p^{e_2}} \left( \a^{[p^{e_2}]}\right) ^{\nu^J_\a(p^{e_1})+1}
\\
&\subseteq 
 \left( \a^{[p^{e_2}]}\right)^{\nu^J_\a(p^{e_1})+1}=\left( \a^{\nu^J_\a(p^{e_1})+1}\right)^{[p^{e_2}]}\\
 & \subseteq  \left( J^{[p^{e_1}]}\right)^{[p^{e_2}]}=  J^{[p^{e_1+e_2}]}.
\end{align*}
Hence, $\nu^J_\a(p^{e_1+e_2})\leq  p^{e_2} \mu(\a) +p^{e_2} \nu^J_\a(p^{e_1})$. The result follows from dividing by $p^{e_1+e_2}.$
\end{proof}

We now show the existence of $F$-thresholds in full generality.

\begin{theorem}\label{Existence c}
Let  $R$ be a ring of prime characteristic $p$. If 
$\a, J\subseteq R$ are ideals such that $\a \subseteq \sqrt{J}$, then 
$\lim\limits_{e\to \infty} \frac{ \nu^J_\a(p^{e})}{p^e}$
exists.
\end{theorem}
\begin{proof}
From Lemma \ref{Bound of Ve} we have 
\[ \frac{\nu^J_\a(p^{e_1+e_2})}{p^{e_1+e_2}}\leq 
\frac{\nu^J_\a(p^{e_1})}{p^{e_1}}+ \frac{\mu(\a)}{p^{e_1}}.\]
Therefore,
\[
\limsup\limits_{e\to\infty} \frac{\nu^J_\a(p^{e})}{p^{e}}=
\limsup\limits_{e_2\to\infty} \frac{\nu^J_\a(p^{e_1+e_2})}{p^{e_1+e_2}}\leq 
\frac{\nu^J_\a(p^{e_1})}{p^{e_1}}+ \frac{\mu(\a)}{p^{e_1}}.\]
Hence,
\[
\limsup\limits_{e\to\infty} \frac{\nu^J_\a(p^{e})}{p^{e}}
\leq 
\liminf_{e_1\to\infty}\left(\frac{\nu^J_\a(p^{e_1})}{p^{e_1}}+ \frac{\mu(\a)}{p^{e_1}}\right)
=\liminf_{e\to\infty}\frac{\nu^J_\a(p^{e})}{p^{e}}.\]
We conclude that 
$\lim\limits_{e\to \infty} \frac{ \nu^J_\a(p^{e})}{p^e}$
exists.

\end{proof}

After Theorem \ref{Existence c}, we can define $F$-thresholds in full generality.

\begin{definition}
Let $R$ be a ring of prime characteristic $p$. Let $\a, J$ be ideals of $R$ such that $\a\subseteq \sqrt{J}$. We define \emph{the $F$-threshold of $\a$ with respect to $J$} by
\[ 
c^J(\a)=\lim_{e\to \infty} \frac{\nu^J_\a(p^e)}{p^e}.
\]  
\end{definition}

We recall some known properties of $F$-thresholds, that we need in what follows.

\begin{proposition}[{\cite[Proposition 2.7]{MTW} \& \cite[Proposition 2.2]{HMTW}}]\label{Prop C}
Let $R$ be a ring of prime characteristic $p$, and let $\a,I,J$ be ideals of $R$. Then

\begin{itemize}
\item[(a)] If $I \supseteq J$ and $\a \subseteq \sqrt{J}$, then $c^I(\a)\leq c^J(\a)$.
\item[(b)] If $\a \subseteq \sqrt{J}$, then $c^{J^{[p]}}(\a)=p \cdot c^{J}(\a)$.
\end{itemize}  
\end{proposition}

\section{A characterization of $\fpt(\a)$}

In this section, we give a characterization of the $F$-pure threshold of a ring as the limit of certain $F$-thresholds. We start by recalling its definition and elementary properties. Before doing this, we need some auxiliary definitions. 

Let $(R,\m,K)$ be either a local ring or a standard graded $K$-algebra. The following ideals, introduced by Aberbach and Enescu \cite{AE},  keep track of the $R$-linear homomorphisms from $F^e_*R$ to $R$ that do not give splittings. For $e \in \NN$, we set
\[ I_e:= \left\{ f\in R  \left| \psi(F^e_*f)\in \m, \text{ for all $R$-linear maps } \psi:F^e_*R\to R \right. \right\}. \] 
In particular, $I_0=\m$.
\begin{remark}
\label{Rem_split_I_e}
Let $(R,\m,K)$ be a local ring. If $f \notin I_e$ for some $e$, then there exists a map $\psi:F^e_*R \to R$ that splits the $R$-module inclusion $F^e_*f \cdot R \subseteq F^e_*R$. When $(R,\m,K)$ is standard graded, the ideals $I_e$ are homogeneous, and the same conclusion is true for a homogeneous element $f \notin I_e$ and homogeneous splitting maps.
\end{remark}

\begin{definition}
Let $(R,\m,K)$ be either a local ring, or a standard graded $K$-algebra.
Suppose that $R$ is an $F$-pure ring. For $e \in \NN$, we associate to the ideals $I_e$ the following integers 
\[ 
b_\a(p^e):=\max \left\{t \in \NN \mid \a^t \not\subseteq I_e \right\}.
\]
Given a proper ideal $\a \subseteq R$, homogeneous when $R$ is graded, we define the \emph{$F$-pure threshold of $\a$ in $R$} as
\[
\fpt(\a):=\lim_{e\to \infty} \frac{b_\a(p^e)}{p^e}.
\]
When $\a=\m$, the $F$-pure threshold $\fpt(\m)$ is often simply denoted by $\fpt(R)$.
\end{definition}

\begin{remark}
The definition presented above is not the original  given  by Takagi and Watanabe \cite{TW2004}. For a real number $\lambda\geq 0$, we say that $(R,\a^\lambda)$ is $F$-pure if for every $e\gg 0$, there exists an element $f\in \a^{\lfloor (p^e-1)\lambda\rfloor}$ such that the inclusion of $R$-modules $F^e_*f \cdot  R\subseteq F^e_* R$ splits.
The original definition of the $F$-pure threshold of $\a$ is
\[ 
\fpt(\a)=\sup\left\{ \lambda\in \RR_{> 0}\ | \ (R,\a^{\lambda})\hbox{ is }F\hbox{-pure} \right\}.\] We refer to \cite[Proposition 3.10]{DSNBFpurity} for a proof that both definitions coincide.
\end{remark}

\begin{remark} \label{Rem Diag Deg Gen}
If $(R,\m,K)$ is a standard graded $F$-pure $K$-algebra, then $b :=\frac{b_\m(p^e)}{p^e}$ is the highest possible degree for a minimal generator of the free part of $F^e_*R$ with the natural $\frac{1}{p^e}\NN$ grading. In other words, if 
$F^e_*R\cong \oplus \left( R(-\gamma_i)\right) \oplus M_e$ as $\frac{1}{p^e}\NN$-graded modules, where $M_e$ is a graded $R$-module with no free summands, then $b=\max\{\gamma_i\}$. This follows from the definition of $b_\m(p^e)$ and by Remark \ref{Rem_split_I_e}.

\end{remark}

\begin{proposition} [{\cite[Lemma 4.4]{TuckerFSig}}] \label{Prop Ie}
Let $(R,\m,K)$ be either a standard graded $K$-algebra or a local ring. Assume that $R$ is $F$-finite. If $R$ is an $F$-pure ring, then $I_e^{[p]}\subseteq I_{e+1}$ for all  $e \in\NN$. 
\end{proposition}

We now present the main result of this section. Namely, we show that the $F$-pure threshold of an ideal is a limit of $F$-thresholds. For principal ideals, this follows from the characterization of the digits in the base $p$-expansion of $\fpt(\a)$  \cite[Key Lemma]{DanielMRL}.

\begin{theorem}\label{Main 1}
Let $(R,\m,K)$ be either a standard graded $K$-algebra or a local ring. Assume further that $R$ is $F$-finite and $F$-pure. If $\a$ is an ideal of $R$, homogeneous in case $R$ is graded, then 
\[ 
\fpt(\a)=\lim_{e\to\infty}\frac{c^{I_e}(\a)}{p^e}.
\]
\end{theorem}

\begin{proof}
By Propositions \ref{Prop Ie} and \ref{Prop C}(a) we have that 
$c^{I_{e+1}}(\a) \leq c^{I_e^{[p]}}(\a)$ for all $e \in \NN$. In addition,  $c^{I_e^{[p]}}(\a)=  p\cdot c^{I_e}(\a)$ by Proposition \ref{Prop C}(b).
Hence, 
\[ 
0 \leq \frac{c^{I_{e+1}}(\a)}{p^{e+1}}\leq \frac{c^{I_e}(\a)}{p^e},
\]
for all $e \in \NN$, which shows the sequence $\{\frac{c^{I_e}(\a)}{p^e}\}_{e \in \NN}$ is non-increasing, and bounded below by zero. As a consequence, it does converge to a limit as $e$ approaches infinity. 

Note that, for all $e \in \NN$, we have $b_\a(p^e)=\max \left\{ t\in \NN \mid  \a^t\not\subseteq I_e \right\}= \nu_\a^{I_e}(p^0)$. Let $s \in \NN$ be an arbitrary integer. By taking $e_1=0$ and $e_2=s$ in Lemma \ref{Bound of Ve}, we deduce that
\[ 
\frac{ \nu^{I_e}_\a(p^s)}{p^s}-b_\a(p^e)\leq \mu (\a) .
\]  
In addition,
$$
0\leq\frac{ \nu^{I_e}_\a(p^s)}{p^s}-b_\a(p^e)\leq \mu (\a),
$$ 
because the sequence $\left\{ \frac{\nu^{I_e}_\a(p^s)}{p^s} \right\}_{e\in \NN}$ is non-decreasing, as $R$ is $F$-pure.
Hence
\[ 
0\leq \frac{ \nu^{I_e}_\a(p^s)}{p^s}-b_\a(p^e)\leq \mu (\a)
\] 
for all $e, s \in \NN$. We take the limit as $s \to \infty$ to get 
\[ 
0\leq c^{I_e}(\a)-b_\a(p^e)\leq \mu(\a),
\] and dividing this expression by $p^e$ gives 
\[ 
0\leq \frac{c^{I_e}(\a)}{p^e}-\frac{b_\a(p^e)}{p^e}\leq \frac{\mu(\a)}{p^e}.
\]
Taking the limit over $e$ gives the result. 
\end{proof}

The previous  result emulates a relation showed by Tucker \cite{TuckerFSig} between the Hilbert-Kunz multiplicity \cite{Monsky} and the $F$-signature \cite{SmithVDB,HLMCM,TuckerFSig}:
$$
s(R)=\lim\limits_{e\to\infty}\frac{\e_{HK}(I_e)}{p^{ed}}.
$$

As a corollary, we obtain a characterization of rings $(R,\m,K)$ for which $\fpt(\a)=c^\m(\a)$, for any ideal $\a \subseteq R$. This gives a first answer to Question \ref{Question Equality}. We study in more details the condition $\fpt(\m) = c^\m(\m)$, for standard graded algebras, in Section \ref{Sec_equality}.

\begin{corollary}\label{Cor fpt=c}
Let $(R,\m,K)$ and $\a \subseteq R$ be as in Theorem \ref{Main 1}. Then, $\fpt(\a)=c^\m(\a)$ if and only if $c^{I_e}(\a)=c^{\m^{[p^e]}}(\a)$ for all integers $e \in\NN$.
\end{corollary}
\begin{proof}
We first assume that $c^{I_e}(\a)=c^{\m^{[p^e]}}(\a)$ for all $e\in \NN$. Then, by Proposition \ref{Prop C} (b), $c^{I_e}(\a)=c^{\m^{[p^e]}}(\a)=p^e c^\m (\a)$, from which we have
\[\fpt(\a)=\lim_{e\to\infty}\frac{c^{I_e}(\a)}{p^e}=\lim_{e\to\infty}\frac{p^ec^{^{\m}}(\a)}{p^e}=c^{\m}(\a).\]

Assume now that $\fpt(\a)=c^\m(\a)$. Since $\left\{\frac{c^{I_e}(\a)}{p^e}\right\}_{e \in \NN}$ forms a non-increasing sequence, with first term $c^\m(\a)$ and limit $\fpt(\a)$, we have that the sequence is constant. Then, $\frac{c^{I_e}(\a)}{p^e}=\fpt(\a)=c^\m (\a)$ for all $e \in \NN$, and thus
$c^{I_e}(\a)=p^e c^\m(\a)=c^{\m^{[p^e]}}(\a)$.
\end{proof}

\section{Limits of $a$-invariants for graded rings}

In this section, we investigate the growth of the $a$-invariants of $R/J^{[p^e]}$ for a homogeneous ideal $J$ over a standard graded algebra $(R,\m,K)$. This study is motivated by the problem of bounding the Castelnuovo-Mumford regularity of $R/J^{[p^e]}$ \cite{KatzmanComplexityFrob,KatzmanZhang} which, in turn, is related to the localization of tight closure at one element \cite{KatzmanComplexityFrob}, the LC condition \cite{HunekeLC,HHLC}, and the discreteness of $F$-jumping numbers \cite{KatzmanZhang}.

\begin{definition}[\cite{GW1}] \label{def a invariant}
Let $(R,\m,K)$ be a standard graded $K$-algebra. Let $M$ be a non-zero $\frac{1}{p^e}\NN$-graded $R$-module. If $H^i_\m(M)\neq 0,$ we define the $i$-th $a$-invariant of $M$ by 
$$
a_i(M)=\sup\left\{s\in \frac{1}{p^e}\ZZ \left| \left[H^i_\m(M)\right]_{s}\neq 0 \right. \right\}.
$$
If $H^i_\m(M)= 0,$ we set $a_i(M)=-\infty$. We define the Castelnuovo-Mumford regularity of $M$ by $\reg(M) = \max\{a_i(M)+i \mid i \in \NN\}$.
\end{definition}

We first present lower bounds for $a$-invariants of $F$-pure rings modulo Frobenius powers of an ideal.

\begin{lemma}\label{lemma ineq frob be}
Let $(R,\m,K)$ be a standard graded $K$-algebra that is $F$-finite and $F$-pure. If $J\subseteq R$ is a homogeneous ideal and $i \in \NN$, then 
\[
\frac{a_i(R/J^{[p^s]})}{p^s}+\frac{b_\m(p^e)}{p^{e+s}}\leq \frac{a_i(R/J^{[p^{e+s}]})}{p^{e+s}}
\]
for all $e,s \in\NN$.
\end{lemma}
\begin{proof}
For every $e \in \NN$, let $b_e=\frac{b_\m(p^e)}{p^e}$. By Remark \ref{Rem Diag Deg Gen}, there exists a $\frac{1}{p^e}\NN$ graded $R$-module, $M_e$, such that
$F^e_*R \cong R(-b_e)\oplus M_e$ as  $\frac{1}{p^e}\NN$ graded $R$-modules.  Let $s$ be a non-negative integer. Applying the functor $- \otimes_R R/J^{[p^s]}$ to the homogeneous split inclusion $R(-b_e) \hookrightarrow F^e_*R$, we obtain that $ \frac{R}{J^{[p^s]}}(-b_e)$ is a direct summand of $\frac{F^e_*R}{J^{[p^s]}F^e_*R}$. Consequently, $H^i_\m(\frac{R}{J^{[p^s]}}(-b_e))$ splits out of $H^i_\m(\frac{F^e_*R}{J^{[p^s]}F^e_*R})$.  Looking at graded components, we conclude
\[
a_i\left(\frac{R}{J^{[p^s]}}\right)+\frac{b_\m(p^e)}{p^{e}}\leq a_i\left(\frac{F^e_*R}{J^{[p^s]}F^e_*R}\right)=\frac{a_i(R/J^{[p^{e+s}]})}{p^{e}}
\]
for all $e,s \in \NN$, where the last step follows from the fact that $\ds \frac{F^e_*R}{J^{[p^s]}F^e_*R} \cong F^{e}_*\left(\frac{R}{J^{[p^{e+s}]}}\right)$. 
Our claim follows after dividing by $p^s.$
\end{proof}

\begin{theorem}\label{Thm Limit a-inv}
Let $(R,\m,K)$ be a standard graded $K$-algebra that is $F$-finite and $F$-pure. Suppose that $J\subseteq R$ is a homogeneous ideal, and let $i \in \NN$. If there exists a constant $C$ such that $a_i(R/J^{[p^e]})\leq Cp^e$ for all $e \in \NN$, then either $H^i_\m(R/J^{[p^s]}) = 0$ for all $s \in \NN$, or
\[
\lim\limits_{e\to \infty} \frac{a_i(R/J^{[p^e]})}{p^e}
\]
exists. Furthermore, we have inequalities
\[
a_i(R/J) + \fpt(\m) \leq \max_{s \in\NN} \left\{\frac{a_i(R/J^{[p^s]})}{p^s}+\frac{\fpt(\m)}{p^s}\right\} \leq \lim\limits_{e\to \infty} \frac{a_i(R/J^{[p^e]})}{p^e}.
\]

\end{theorem}
\begin{proof}
Assume that $H^i_\m(R/J^{[p^e]}) \ne 0$ for some $e \in\NN$. By Lemma \ref{lemma ineq frob be}, for any $s \in\NN$ we have 
\begin{equation} \tag{1} \label{ineq_1}
\ds \frac{a_i(R/J^{[p^s]})}{p^s} + \frac{\fpt(\m)}{p^s} \leq \liminf\limits_{e\to \infty} \frac{a_i(R/J^{[p^{e+s}]})}{p^{e+s}} = \liminf\limits_{e \to \infty} \frac{a_i(R/J^{[p^e]})}{p^e}.
\end{equation}

In particular, since there exists $e \in\NN$ for which $a_i(R/J^{[p^e]}) > - \infty$, we have that $\liminf\limits_{e\to \infty} \frac{a_i(R/J^{[p^e]})}{p^e}$ is finite. By assumption, we have that $a_i(R/J^{[p^e]})\leq Cp^e$ for all $e \in \NN$. Therefore, $\limsup\limits_{e\to \infty} \frac{a_i(R/J^{[p^e]})}{p^e}$ is also finite. Taking limsup with respect to $s$ in (\ref{ineq_1}) gives 
\[
\limsup\limits_{s\to \infty}\frac{a_i(R/J^{[p^{s}]})}{p^{s}}\leq \liminf\limits_{e \to \infty}\frac{a_i(R/J^{[p^{e}]})}{p^{e}},
\]
which implies that the limit exists. The last claim now follows from the fact that
\[
\ds \frac{a_i(R/J^{[p^s]})}{p^s} + \frac{\fpt(\m)}{p^s} \leq \lim\limits_{e\to \infty} \frac{a_i(R/J^{[p^{e+s}]})}{p^{e+s}} = \lim_{e \to \infty} \frac{a_i(R/J^{[p^e]})}{p^e}
\]
for all $s \in\NN$, by Lemma \ref{lemma ineq frob be}.
\end{proof}

\begin{theorem}\label{Thm Limit reg-inv}
Let $(R,\m,K)$ be a standard graded $K$-algebra that is $F$-finite and $F$-pure. Suppose that $J\subseteq R$ is a homogeneous ideal.
If  there exists a constant $C$ such that $\reg(R/J^{[p^e]})\leq Cp^e$ for all $e \in \NN$, 
then 
\[
\lim\limits_{e\to \infty} \frac{\reg(R/J^{[p^e]})}{p^e}.
\]
exists, and it is bounded below by
$\max\{a_i(R/J) \ | \ i \in \NN\}+\fpt(\m).$ 
\end{theorem}
\begin{proof}
Let $\mathcal{I} = \{i \in \NN \mid H^i_\m(R/J^{[p^s]}) \ne 0$ for some $s \in \NN\}$. Note that, for all $e \in \NN$, we have $\reg(R/J^{[p^e]}) =  \max_{i \in \mathcal{I}} \{a_i(R/J^{[p^e]}) + i\}$. In addition, since $a_i(R/J^{[p^e]})\leq \reg(R/J^{[p^e]})\leq Cp^e$ for all $e \gg 0$, by Theorem \ref{Thm Limit a-inv} we obtain that  $\lim\limits_{e\to \infty} \frac{a_i(R/J^{[p^e]})}{p^e}$ exists for all $i \in \mathcal{I}$. We have that
\begin{align*}
\lim\limits_{e\to \infty} \frac{\reg(R/J^{[p^e]})}{p^e}
& =\lim\limits_{e\to \infty} \frac{\max_{i \in \mathcal{I}}\{a_i(R/J^{[p^e]})+i\} }{p^e}\\
& =\lim\limits_{e\to \infty}\max_{i \in \mathcal{I}} \left\{ \frac{a_i(R/J^{[p^e]})+i }{p^e}\right\}\\
& =\max_{i \in \mathcal{I}} \left\{\lim\limits_{e\to \infty} \frac{a_i(R/J^{[p^e]})+i }{p^e}\right\}\\
& =\max_{i \in \mathcal{I}} \left\{\lim\limits_{e\to \infty} \frac{a_i(R/J^{[p^e]}) }{p^e}\right\}.
\end{align*}
Therefore $\lim_{e \to \infty} \frac{\reg(R/J^{[p^e]})}{p^e}$ exists. The claim about the lower bound follows from the inequality in Theorem \ref{Thm Limit a-inv}.
\end{proof}

In Theorems \ref{ThmRegDimR/J_t} and \ref{ThmRegDimR/J_t-1} below, we recover linear upper bounds for $a_i(R/J^{[p^e]})$ when $i \geq \dim(R/J)-1$. This type of bound has already been discovered by  Zhang \cite[Corollary 1.3]{ZhangRegFrob}. However, our proof is more direct and does not make use of spectral sequences. In addition, we prove the existence of  $\lim\limits_{e \to \infty} \frac{a_i(R/J^{[p^e]})}{p^e}$ for $i\geq \dim(R/J)-1$, and we obtain specific lower and upper bounds for the limits. 

\begin{remark}[{\cite[Remark 4.8]{DSNBFpurity}}]\label{Rem Diag Deg Gen}
If $(R,\m,K)$ is a standard graded $F$-pure $K$-algebra, then $\frac{\nu^\m_\m(p^e)}{p^e}$ is the highest possible degree of a minimal homogeneous generator of $F^e_*R$, with the natural $\frac{1}{p^e}\NN$ grading. More specifically, we have that
$$
\frac{\nu^\m_\m(p^e)}{p^e}=\sup\left\{s\in\frac{1}{p^e}\NN \ \bigg| \left[ \frac{F^e_*R}{\m F^e_*R} \right]_{s}\neq 0\right\}.
$$
\end{remark}

\begin{lemma}\label{lemma ineq frob nu}
Let $(R,\m,K)$ be an $F$-finite standard graded $K$-algebra. Let $J\subseteq R$ be a homogeneous ideal, and let $t=\dim(R/J)$.
Then, 
$$
\frac{a_t(R/J^{[p^e]})}{p^e}
\leq a_t(R/J)+\frac{\nu^\m_\m(p^e)}{p^e}
$$
for every $e\in\NN.$
\end{lemma}
\begin{proof}
Let $u_1,\ldots,u_\ell$ be a minimal set of homogeneous generators for $F^e_*R$, with degrees $\gamma_1\leq \ldots\leq \gamma_\ell.$
We note that $\gamma_\ell =\frac{\nu^\m_\m(p^e)}{p^e}$ by Remark \ref{Rem Diag Deg Gen}.
Consider the homogeneous surjection $\bigoplus R(-\gamma_i)\twoheadrightarrow  F^e_*R$ obtained from this minimal generating set. Since tensor product is right exact, we have a surjection $\bigoplus R/J(-\gamma_i)\twoheadrightarrow  F^e_*R/JF^e_*R$ of $R/J$-modules that, in turn, induces a surjective map
$\bigoplus H^t_\m ( R/J(-\gamma_i))\twoheadrightarrow  H^t_\m ( F^e_*R/J F^e_*R)$. We note that $H^t_\m(R/J)\neq 0$. 
Then,
$$
\frac{a_t(R/J^{[p^e]})}{p^e}= a_t\left(\frac{F^e_*R}{JF^e_*R}\right)\leq a_t\left(R/J\right)+\gamma_\ell,$$
where the first step follows from the fact that $\ds\frac{F^e_*R}{JF^e_*R}\cong F^e_*\left(\frac{R}{J^{[p^e]}}\right)$.
\end{proof}

\begin{theorem}\label{ThmRegDimR/J_t}
Let $(R,\m,K)$ be an $F$-finite standard graded $K$-algebra. Let $J\subseteq R$ be a homogeneous ideal, and let $t=\dim(R/J)$.
Then, there exists a constant $A$ 
such that $a_t(R/J^{[p^e]})\leq A p^e$  for all $e \in \NN$. In addition, if $R$ is $F$-pure, the limit
\[
\lim\limits_{e\to \infty} \frac{a_t(R/J^{[p^e]})}{p^e} \hspace{1cm} 
\]
exists.  Moreover, we have that
\[
a_t(R/J)+\fpt(\m)\leq \lim\limits_{e\to \infty} \frac{a_t(R/J^{[p^e]})}{p^e}\leq a_t(R/J)+ c^\m(\m).
\]
\end{theorem}

\begin{proof} 
By Lemma \ref{lemma ineq frob nu}, for all $e \in \NN$ we have an inequality
\[
\frac{a_t(R/J^{[p^e]})}{p^e}
\leq a_t(R/J)+\frac{\nu^\m_\m(p^e)}{p^e}.
\]
Since $\{\frac{\nu^\m_\m(p^e)}{p^e}\}_{e\in\NN}$ converges by Theorem \ref{Existence c}, the sequence is bounded. Let $A'$ be any upper bound. If we let $A=A'+a_t(R/J)$, then $a_t(R/J^{[p^e]})\leq A p^e$ for all $e \in \NN$, as desired.
If $R$ is assumed to be $F$-pure, then $\lim\limits_{e\to \infty} \frac{a_t(R/J^{[p^e]})}{p^e}$ exists by Theorem \ref{Thm Limit a-inv}. In addition, Theorem \ref{Thm Limit a-inv} and Lemma \ref{lemma ineq frob nu} yield the following inequalities
\[
\ds  a_t(R/J)+\fpt(\m)\leq \lim\limits_{e\to \infty} \frac{a_t(R/J^{[p^e]})}{p^e}\leq a_t(R/J)+ c^\m(\m).
\]
\end{proof}

\begin{theorem} \label{ThmRegDimR/J_t-1} Let $(R,\m,K)$ be an $F$-finite standard graded $K$-algebra. Let $J\subseteq R$ be a homogeneous ideal, and let $t=\dim(R/J)$. Then, there exists a constant $B$ such that $a_{t-1}(R/J^{[p^e]})\leq B p^e$ for all $e\in \NN$.
In addition, if $R$ is $F$-pure and $H^{t-1}_\m(R/J^{[p^e]}) \ne 0$ for some $e \in \NN$, then 
\[
\lim\limits_{e\to \infty} \frac{a_{t-1}(R/J^{[p^e]})}{p^e}
\]
exists. Furthermore, we have that
\[
a_{t-1}(R/J)+\fpt(\m)\leq \lim\limits_{e\to \infty} \frac{a_{t-1}(R/J^{[p^e]})}{p^e}\leq D(c^J(J)+1),
\]
where $D=\max\{t \in \NN \;|\;\left[ J/\m J\right]_t \neq 0\}+1$.
\end{theorem}
\begin{proof}
Since $\{\frac{\nu^J_J(p^e)}{p^e}\}_{e\in\NN}$ converges by Theorem \ref{Existence c}, there exists an integer $B'$ such that $\nu^J_J(p^e)\leq B'p^e$ for all $e \in \NN$, yielding inclusions $J^{B'p^e+1} \subseteq J^{\nu^J_J(p^e)+1} \subseteq J^{[p^e]}$. For each $e \in \NN$, the short exact sequence
\[
\xymatrixcolsep{5mm}
\xymatrixrowsep{2mm}
\xymatrix{
0 \ar[r] & \ds J^{[p^e]}/J^{B'p^e+1} \ar[r] & R/J^{B'p^e+1} \ar[r] & R/J^{[p^e]} \ar[r] & 0
}
\]
induces the following exact sequence in local cohomology
\[
\xymatrixcolsep{5mm}
\xymatrixrowsep{2mm}
\xymatrix{
H^{t-1}_\m(R/J^{B'p^e+1}) \ar[r] & H^{t-1}_\m(R/J^{[p^e]}) \ar[r] & H^t_\m(J^{[p^e]}/J^{B'p^e+1}).
}
\]
For ordinary powers of an ideal, an explicit linear upper bound for the Castelnuovo-Mumford regularity in terms of the degree of minimal generators is known \cite[Theorem 3.2]{TrungWang} (see also \cite{CHT,Kodiyalam}). As a consequence of \cite[Theorem 3.2]{TrungWang}, there exists a constant $C$ such that $a_{t-1}(R/J^{B'p^e+1}) \leq D'(B'p^e+1)+C$ for all $e \gg 0$, where $D'=D-1$ is the maximal degree of a minimal homogeneous generator of $J$ . Since $B'p^e+1 \geq C$ for $e$ large enough, we deduce that $a_{t-1}(R/J^{B'p^e+1}) \leq D(B'p^e+1)$ for all $e \gg 0$. For a given $e \in \NN$, if $\dim(J^{[p^e]}/J^{B'p^e+1})<t$ then $H^t_\m(J^{[p^e]}/J^{B'p^e+1}) = 0$, and we obtain 
\[
\ds a_{t-1}(R/J^{[p^e]}) \leq a_{t-1}(R/J^{B'p^e+1}) \leq D(B'p^e+1).
\]
On the other hand, if $\dim(J^{[p^e]}/J^{B'p^e+1}) = t$ and $e \gg 0$, the short exact sequence above gives
\begin{align*} \label{ineq1} \tag{1}
 a_{t-1}(R/J^{[p^e]})& \leq \max\{ a_{t-1}(R/J^{B'p^e+1}), a_t(J^{[p^e]}/J^{B'p^e+1}) \} \\
& \leq \max\{D(B'p^e+1),a_t(J^{[p^e]}/J^{B'p^e+1})\}.
\end{align*}
We claim that the inequality $a_t(J^{[p^e]}/J^{B'p^e+1}) \leq D(B'+1)p^e+Dp^e$ holds true for all $e \gg 0$. Let $f_1,\ldots,f_s$ be minimal homogeneous generators of $J$, of degrees $d_1,\ldots,d_s$. Note that $D=\max\{d_j \;|\; j=1,\ldots,s\} +1$. In addition, note that $J^{[p^e]}/J^{B'p^e+1}$ is a $R/J^{B'p^e+1}$ module generated by the residue classes of $f_1^{p^e},\ldots,f_s^{p^e}$. We have the following surjection, that is homogeneous of degree zero: 
\[
\xymatrixcolsep{5mm}
\xymatrixrowsep{2mm}
\xymatrix{
\ds \bigoplus_{j=1}^s \frac{R}{J^{B'p^e+1}}(-d_jp^e) \ar@{->>}[r] & \ds \frac{J^{[p^e]}}{J^{B'p^e+1}}.
}
\]
Taking local cohomology we obtain a surjection 
\[
\xymatrixcolsep{5mm}
\xymatrixrowsep{2mm}
\xymatrix{
\ds \bigoplus_{j=1}^s H^t_\m\left( \frac{R}{J^{B'p^e+1}}(-d_jp^e) \right) \ar@{->>}[r] & \ds H^t_\m\left( \frac{J^{[p^e]}}{J^{B'p^e+1}}\right)
}
\]
which, in turn, gives
\begin{align*} \label{ineq2} \tag{2}
\ds a_t(J^{[p^e]}/J^{B'p^e+1}) & \leq \max\{a_t(R/J^{B'p^e+1})+d_jp^e \mid 1 \leq  j \leq s \} \\ 
& \leq D(B'p^e+1)+Dp^e,
\end{align*}
as claimed. Putting (\ref{ineq1}) and (\ref{ineq2}) together, we conclude that there exists $e_0 \in \NN$ such that $a_{t-1}(R/J^{[p^e]}) \leq D(B'p^e+1)+Dp^e \leq D(B'+2)p^e$ for all $e > e_0$. Taking 
\[
\ds B=\max\left\{D(B'+2),\frac{a_{t-1}(R/J^{[p^s]})}{p^s} \mid 0 \leq s \leq e_0\right\}
\]
we finally obtain that $a_{t-1}(R/J^{[p^e]})\leq  Bp^e$ for all $e \in \NN$, as desired. 

Now assume that $R$ is $F$-pure and $H^{t-1}_\m(R/J^{[p^e]}) \ne 0$ for some $e \in \NN$. Theorem \ref{Thm Limit a-inv} implies that $\lim \limits_{e \to \infty} \frac{a_{t-1}(R/J^{[p^e]})}{p^e}$ exists, and gives the lower bound for the limit. For the upper bound, note that we can set $B'=c^J(J)$, because $c^J(J) \geq \frac{\nu_J^J(p^e)}{p^e}$ for all $e \in \NN$ for $F$-pure rings. With this choice of $B'$, combining (\ref{ineq1}) and (\ref{ineq2}), we deduce that $a_{t-1}(R/J^{[p^e]}) \leq D(c^J(J)p^e+1) + Dp^e = D(c^J(J)+1)p^e + D$ for all $e \gg 0$. This gives
\[
\ds \lim_{e \to \infty} \frac{a_{t-1}(R/J^{[p^e]})}{p^e} \leq \lim_{e \to \infty} \frac{D(c^J(J)+1)p^e + D}{p^e} = D(c^J(J)+1).
\]
\end{proof}

\section{The equality $\fpt(\m) = c^{\m}(\m)$ for  standard graded rings} \label{Sec_equality}

In this section we prove that, for a Gorenstein standard graded algebra $(R,\m,K)$,  the equality between $\fpt(\m)$ and the so-called diagonal $F$-threshold $c^\m(\m)$ implies that $R$ is strongly $F$-regular. Throughout this section we assume that $(R,\m,K)$ is a standard graded ring. We start by making some observations about $a$-invariants for rings satisfying $\fpt(\m) = c^{\m}(\m)$.

\begin{proposition}\label{Prop equi frob a-inv}
Let $(R,\m,K)$ be a standard graded $d$-dimensional $K$-algebra that is $F$-finite and $F$-pure. Let $J\subseteq R$ be a homogeneous ideal, and let $t=\dim(R/J)$.
If $\fpt(\m)=c^\m(\m)$, then for all $e \in \NN$ we have
$$
\frac{a_t(R/J^{[p^e]})}{p^e}
= a_t(R/J)-a_d(R)+\frac{a_d(R)}{p^e}.$$
\end{proposition}
\begin{proof}
In our assumptions, $\fpt(\m)=-a_d(R)=c^\m(\m)$ \cite[Theorem B]{DSNBFpurity}. From Lemmas \ref{lemma ineq frob be} and \ref{lemma ineq frob nu}, for all homogeneous ideals $\a \subseteq R$ we obtain that
\[
\lim\limits_{s\to\infty} \frac{a_{\dim(R/\a)}(R/\a^{[p^s]})}{p^s}
= a_{\dim(R/\a)}(R/\a)-a_d(R).
\]
In particular, choosing $\a=J^{[p^e]}$ and dividing by $p^e$, this implies that
$$
\frac{a_{t}(R/J^{[p^e]})-a_d(R)}{p^e}=
\lim\limits_{s\to\infty} \frac{a_{t}(R/J^{[p^{(e+s)}]})}{p^{(e+s)}}=
\lim\limits_{s\to\infty} \frac{a_{t}(R/J^{[p^{s}]})}{p^{s}}
= a_{t}(R/J)-a_d(R).
$$
Hence, for every $e \in \NN$, we have $\frac{a_{t}(R/J^{[p^e]})}{p^e}=a_{t}(R/J)-a_d(R)+\frac{a_d(R)}{p^e}$.
\end{proof}

\begin{remark}\label{Rem a-inv nu}
Suppose that $(R,\m,K)$ is a standard graded $K$-algebra. If $J$ is an $\m$-primary homogeneous ideal, then for all $e \in \NN$
$$\nu^J_\m(p^e)=\max\{s \in \NN \mid \m^s\not\subseteq J^{[p^e]}\}=\max\{ s \in \NN\mid \m^s (R/J^{[p^e]})\neq 0\}=a_0(R/J^{[p^e]}).
$$
\end{remark}

The following corollary gives a formula to compute $c^J(\m)$
in terms of certain $a$-invariants.

\begin{corollary}\label{Cor fpt=c formula}
Let $(R,\m,K)$ be a standard graded $d$-dimensional $K$-algebra that is $F$-finite and $F$-pure. Let $J\subseteq R$ be a homogeneous $\m$-primary ideal.
If $\fpt(\m)=c^\m(\m),$
then 
$$
\frac{ \nu^J_\m(p^e)}{p^e}
= a_0(R/J)-a_d(R)+\frac{a_d(R)}{p^e}.
$$
In particular,
$c^J(\m)= a_0(R/J)-a_d(R).$
\end{corollary}
\begin{proof}
This follows immediately from Proposition \ref{Prop equi frob a-inv}, Remark \ref{Rem a-inv nu}, and the fact that $\dim(R/J)=0$.
\end{proof}

We recall the definition of compatible ideals, which play an important role in showing that the equality $\fpt(\m) = c^{\m}(\m)$ implies that $R$ is a domain for Gorenstein rings. 

\begin{definition} [{\cite{KarlCentersFpurity}}] Let $(R,\m,K)$ be a reduced $F$-finite standard graded $K$-algebra. An ideal $J \subseteq R$ is said to be \emph{compatible} if $\varphi(F^e_*J) \subseteq J$ for all integers $e \geq 1$ and all $R$-homomorphisms $\varphi \in \Hom_R(F^e_*R,R)$.
\end{definition}

\begin{lemma}\label{Lemma Comp}
Let $(R,\m,K)$ be a $d$-dimensional standard graded Gorenstein $K$-algebra that is $F$-finite and $F$-pure. Let $J\subseteq R$ be a homogeneous compatible ideal, $\n=\m (R/J)$, and $t=\dim(R/J)$.
If $\fpt(\m)=c^\m(\m),$ then $\fpt(\m)=\fpt(\n)=c^\n (\n)=c^\m(\m)$. In particular, $a_d(R)=a_t(R/J).$
\end{lemma}
\begin{proof}
We note that $\m^r \subseteq \m^{[p^e]}$ implies that $\n^r=\m^r(R/J)\subseteq \m^{[p^e]} (R/J)=\n^{[p^e]}$. As a consequence, we have that $c^{\n}(\n)\leq c^{\m}(\m)$. We also know that $\fpt(\m)\leq \fpt(\n)$ \cite[Theorem 4.7]{DSNBFpurity}.
Since $\fpt(\m)=c^\m(\m)$, we obtain $\fpt(\m) = \fpt(\n)=c^\n (\n) =c^\m(\m)$.
The last statement follows from the fact that $\fpt(\m)\leq -a_d(R)\leq c^\m(\m)$ and $\fpt(\n)\leq -a_t(R/J)\leq c^\n(\n)$ 
\cite[Theorem B]{DSNBFpurity}.
\end{proof}

The following lemma is a key ingredient in the proof of Theorem \ref{main thm}. 
\begin{lemma} \label{Lemma domain}
Let $(R,\m,K)$ be a standard graded Gorenstein $K$-algebra that is $F$-finite and $F$-pure. If $\fpt(\m)=c^\m(\m),$ then $R$ is a domain.
\end{lemma}
\begin{proof}
We proceed by way of contradiction. Let $Q_1,\ldots,Q_\ell$ be the minimal primes of $R$. Since  $R$ is not a domain, $\ell\geq 2.$ We set $J=Q_2\cap\ldots\cap Q_\ell$, and we note that $Q_1$ and $J$ are compatible ideals \cite[Corollary 4.8 \& Lemma 3.5]{KarlCentersFpurity}. Furthermore, we have $d:=\dim(R)=\dim(R/Q_1)=\dim(R/J),$ because $R$ is a standard graded Cohen-Macaulay $K$-algebra, hence it is equidimensional. In addition, since $R$ is $F$-pure, it is reduced. There is a short exact sequence
\[
0\to R\to R/Q_1\oplus R/J\to R/(Q_1+J)\to 0,
\]
which induces a long exact sequence on local cohomology
$$
\ldots\to H^{d-1}_{\m}(R/(Q_1+J))\to H^{d}_{\m}(R)\to H^{d}_{\m}(R/Q_1)\oplus H^{d}_{\m}(R/J)\to 0.
$$
We point out that $H^{d}_{\m}(R/(Q_1+J))=0$ because $\dim R/(Q_1+J)\leq d-1.$
Let $a=a_d(R).$
Then
$$\left[ H^{d}_{\m}(R)\right]_{a}\to \left[  H^{d}_{\m}(R/Q_1)\right]_{a}\oplus \left[ H^{d}_{\m}(R/J)\right]_{a}$$
is surjective, and thus
\begin{align*}
1&= \dim_K \left[ H^{d}_{\m}(R)\right]_{a}\hbox{ because }R\hbox{ is Gorenstain.} \\
&\geq \dim_K\left[  H^{d}_{\m}(R/Q_1)\right]_{a}+\dim_K\left[ H^{d}_{\m}(R/J)\right]_{a}\\
& \geq 2\hbox{ because }a=a_d(R/Q_1)=a_d(R/J)\hbox{ by Lemma \ref{Lemma Comp}}.
\end{align*}
Hence, we get a contradiction, and $R$ must be a domain.
\end{proof}

The following lemma allows us to reduce to the case of an infinite coefficient field. 
\begin{lemma} \label{lem: Inf residue field}
Let $(R,\m,K)$ be a standard graded $d$-dimensional Gorenstein $K$-algebra that is $F$-finite and $F$-pure. Let $\overline{K}$ be the algebraic closure of $K$ and $\overline{\m}$ be the irrelevant maximal ideal of the ring $R \otimes_K \overline{K}$. Then, $R\otimes_K \overline{K}$ is also a Gorenstein $F$-pure ring, $\fpt(\m)=\fpt(\overline{\m})$, and $c^{\m}(\m)=c^{\overline{\m}}(\overline{\m})$.  
\end{lemma}
\begin{proof}
The map $R \to R\otimes_K \overline{K}=: \overline{R}$ is faithfully flat, and the irrelevant maximal ideal $\m$ of $R$ extends to the irrelevant maximal ideal $\overline{\m}$ of $\overline{R}$ under such extension. It follows that $\overline{R}$ is Cohen-Macaulay of dimension $d$, and has the same type as $R$. Hence $\overline{R}$ is a Gorenstein standard graded $\overline{K}$-algebra. Moreover, we have that $\overline{R}$ is an $F$-pure ring as consequence of Fedder's Criterion, Theorem \ref{Fedder}, since colon ideals and non-containments are preserved under faithfully flat extensions. We also note that $a_d(R)=a_d(R\otimes_K \overline{K})$, because there is a graded isomorphism $H^d_\m(R) \otimes_K \overline{K} \cong H^d_{\overline{\m}}(R \otimes_K \overline{K})$. By these observations and the fact that  $R$ and $\overline{R}$ are Gorenstein standard graded $K$-algebras, we see that $\fpt(\m)=\fpt(\overline{\m})$, since in this case the $F$-pure thresholds coincide with the respective $a$-invariants \cite[Theorem B]{DSNBFpurity}. In addition, $\nu^\m_\m(p^e) = \nu_{\overline{\m}}^{\overline{\m}}(p^e)$ for all $e \in \NN$, therefore $c^\m(\m)=c^{\overline{\m}}(\overline{\m})$.
\end{proof}

Theorem \ref{main thm} establishes the strong $F$-regularity of $R$, as well as a lower bound of the $F$-signature. The $F$-signature is an invariant that measures how far a strongly $F$-regular ring is from being regular. We start by recalling the definition of a couple of invariants in positive characteristic. We restrict ourselves to the standard graded setting.

We now recall the definitions and concepts that are relevant towards presenting Theorem \ref{main thm}. We restrict ourselves to the standard graded setting, since this is the level of generality in which we work for the rest of the article. We refer the reader to \cite{HoHu1,HunekeSurvey} for more general definitions and statements.

\begin{definition}[\cite{HoHuStrong, AE}] Let $(R,\m,K)$ be an $F$-finite standard graded $K$-algebra. We say that $R$ is strongly $F$-regular if, for all homogeneous elements $c \ne 0$, there exists $e \gg 0$ such that $c \notin I_e$. Equivalently, $R$ is strongly $F$-regular if $\bigcap_{e \in \NN} I_e = (0)$.
\end{definition} 

We point out that there are several characterizations of strong $F$-regularity. 
The original definition given by Hochster and Huneke is in terms of existence of splitting maps. The definition we give is equivalent in view of Remark \ref{Rem_split_I_e}. 
For the purposes of this article, it is helpful to recall an equivalent formulation in terms of the (big) test ideal $\tau(R)$. Namely, a ring $R$ is strongly $F$-regular if and only if $\tau(R) = R$ \cite[Theorem 7.1 (5)]{LyuKaren}. 

Another characterization of strong $F$-regularity can be given in terms of the $F$-signature, which we now introduce formally in the graded setup.

\begin{definition}[{\cite{SmithVDB,HLMCM,TuckerFSig}}]
Let $(R,\m,K)$ be a $d$-dimensional $F$-finite standard graded $K$-algebra. \emph{The $F$-signature of $R$} is defined by
$$
s(R)=\lim\limits_{e\to\infty}\frac{\lambda(R/I_e)}{p^{ed}}.
$$
\end{definition}
One can show that $s(R)$ equals the $F$-signature $s(R_\m)$ of the local ring $R_\m$. In addition, in our assumptions, $s(R)$ also coincides with the global $F$-signature of $R$ 
\cite{GlobalFrobBetti}. The $F$-signature is an important invariant for rings of positive characteristic.  For example, $R$ is regular if and only if $s(R)=1$ \cite[Corollary 16]{HLMCM}, and $R$ is strongly $F$-regular if and only if $s(R)>0$ \cite[Theorem 0.2]{ALFSig}. See also \cite[Theorem B]{GlobalFinv} for a global version of these results.

\begin{definition}[{\cite{Monsky}}]
Let $(R,\m,K)$ be a $d$-dimensional standard graded $K$-algebra, and let $J$ be an $\m$-primary homogeneous ideal.
The Hilbert-Kunz multiplicity of $J$ is defined by 
$$
\e_{HK}(J)=\lim\limits_{e\to\infty}\frac{\lambda(R/J^{[p^e]})}{p^{ed}}.
$$
\end{definition} 
This invariant measures the singularities of a ring. For instance, $R$ is regular if and only if it is formally unmixed and $\e_{HK}(\m) = 1$ \cite{WY}. Furthermore, smaller values of $e_{HK}(\m)$ typically imply better properties of the ring \cite{BlickleEnescu,AELower_eHK}.

\begin{remark}[{\cite[Proof of Theorem 11]{HLMCM}}]\label{RemFsigGor}
If $(R,\m,K)$ is a Gorenstein graded algebra, $J$ is a homogeneous system of parameters, and $\a=(J:_R\m)$, we have that $s(R)=\e_{HK}(J)-\e_{HK}(\a).$
\end{remark}

\begin{definition}
Let $(R,\m,K)$ be a standard graded $K$-algebra of dimension $d$. We denote by $e(R)$ the Hilbert Samuel multiplicity of the irrelevant maximal ideal $\m$ in $R$, that is
\[
\ds e(R) = \lim_{n \to \infty} \frac{d!}{n^d} \lambda(R/\m^n)
\]
\end{definition}

When $(R,\m,K)$ is Cohen-Macaulay, $e(R) = e(J) = \lambda(R/J)$ for any homogeneous ideal $J$ that is a minimal reduction of $\m$.

We are now ready to present the main result of this section.  For this, we need the preparatory results obtained in this section and, as a crucial tool, we invoke a characterization of tight closure and integral closure for parameter ideals in terms of $F$-thresholds \cite[Section 3]{HMTW}.

\begin{theorem} \label{main thm}
Let $(R,\m,K)$ be a $d$-dimensional standard graded Gorenstein $K$-algebra that is $F$-finite and $F$-pure. If $\fpt(\m)=c^\m(\m)$, then $R$ is strongly $F$-regular. In addition,
$$s(R) \geq \frac{\e(R)}{d!}.$$
\end{theorem}
\begin{proof}

Let $\overline{K}$ be the algebraic closure of $K$.
If $R\otimes_K \overline{K}$ is strongly $F$-regular, then so is $R$ \cite[Corollary 3.8]{AEBaseChange}. Furthermore, since the closed fiber of the extension $R\to R\otimes_K \overline{K}$ is regular, we have $s(R)\geq s(R\otimes_K \overline{K})$ \cite[Theorem 5.4]{YaoObsFsig}. In light of Lemma \ref{lem: Inf residue field}, we can assume that $K = \overline{K}$ and, in particular, that $K$ is infinite. As a consequence, there exists a homogeneous ideal $J$, generated by a homogeneous system of parameters, such that $\overline{J}=\m$. If we let $\a=J:_R \m$, then $a_0(R/\a)\leq a_0(R/J)-1$, and we have 
\begin{align*}
c^\a(J)& = c^\a(\m)\\
&= a_0(R/\a)-a_d(R)\hbox{ by Corollary \ref{Cor fpt=c formula}.}\\
&\leq a_0(R/J)-a_d(R)-1\\
&=c^J(\m)-1.
\end{align*}

Let $T=\widehat{R_\m}.$ We note that $T$ is a domain because the associated graded ring $\gr_\m(T)=\oplus_i (\m T)^i/(\m T)^{i+1}$ is isomorphic to $R$, which is a domain. 
In this case, $c^\a (J)=c^{\a T}(JT)$ and 
$c^J(\m)=c^{JT}(\m T)$, because $\a,J,$ and $\m$ are $\m$-primary ideals.
Then,
$c^{\a T}(J T)=c^\a (J)\leq c^J(\m)-1=c^{JT}(\m T)-1\leq d-1,$ 
\cite[Theorem 3.3]{HMTW}, and it follows that  $\a T\not\subseteq (JT)^*$ \cite[Corollary 3.2]{HMTW}. Since $T$ is Gorenstein, this means that the socle of $JT$ does not intersect $(JT)^*$, and we conclude that $JT=(JT)^*$. In addition, for a Gorenstein ring, a parameter ideal being tightly closed is an equivalent condition to being strongly $F$-regular \cite{HH-MSMF}  (see also \cite{CraigBookTC}).
Finally, since the test ideal commutes with localization and completion for Gorenstein rings \cite[Theorem 7.1]{LyuKaren}, we have that $T=\tau(T)=\tau(R) T$. Because the test ideal is a homogeneous ideal \cite[Lemma 4.2]{FregEquiv}, we obtain that $\tau(R)=R$, hence $R$ is a strongly $F$-regular ring.

We now focus on proving the inequality involving the $F$-signature. With the same reductions and the same notation introduced in the first part of the proof, recall that $c^\a(J)\leq d-1$.
In addition, since $R$ is $F$-pure, note that $\frac{\nu^\a_J(p^e)}{p^e}\leq c^\a(J)$ for all $e \in \NN$.
Then, for all non-negative integers $e$, we have a series of containments
\[
J^{p^e(d-1)+1}\subseteq J^{p^ec^{\a}(J)+1}\subseteq J^{\nu^{\a}_J(p^e)+1}\subseteq \a^{[p^e]},
\]
and therefore we get $\lambda((J^{p^e(d-1)+1}+J^{[p^e]})/J^{[p^e]})\leq \lambda(\a^{[p^e]}/J^{[p^e]}).$
Consider the set 
\[
\cA_e=\{(\alpha_1,\ldots,\alpha_d) \in\NN^d\;|\; \alpha_i\leq p^e-1 \hbox{ for all }i\quad \& \quad \alpha_1+\ldots+\alpha_d\geq p^e(d-1)+1 \}.
\]
Since $R$ is Cohen-Macaulay and $J$ is a parameter ideal, $J$ is generated by a regular sequence $\underline{f}=f_1,\ldots, f_d.$ The monomials in $\underline{f}$ with exponents in $\cA_e$ induce a filtration  on $(J^{p^e(d-1)+1}+J^{[p^e]})/J^{[p^e]}$:
\[
\ds 0=N_0\subseteq N_1\subseteq \ldots \subseteq N_{|\cA_e|}=(J^{p^e(d-1)+1} + J^{[p^e]})/J^{[p^e]},
\]
with the property that $N_{i+1}/N_i\cong R/J$ for all $0 \leq i \leq |\cA_e|-1$. As a consequence, for all $e \in \NN$ we see that 
$
\lambda((J^{p^e(d-1)+1}+ J^{[p^e]})/J^{[p^e]})=\lambda(R/J)|\cA_e|$.
We note that the set 
\[
\ds \cB_e=\{(\beta_1,\ldots,\beta_d) \in\NN^d \mid \beta_1+\ldots+\beta_d\leq p^e-d-1 \}
\]
is in bijective correspondence with $\cA_e$ via the function $\varphi:\cA_e\to\cB_e$, given by $\varphi((\alpha_1,\ldots,\alpha_d))=(p^e-1-\alpha_1,\ldots,p^e-1-\alpha_d).$
Since $|\cA_e|=|\cB_e|=\binom{p^e-1}{d}$, we obtain the following relations:
\begin{align*}
s(R)&=\e_{HK}(\a)-\e_{HK}(J)\hbox{ by Remark \ref{RemFsigGor}}\\
&=\lim\limits_{e\to\infty}\frac{\lambda(R/\a^{[p^e]})}{p^{ed}}-\lim\limits_{e\to\infty}\frac{\lambda(R/J^{[p^e]})}{p^{ed}} \\
&=\lim\limits_{e\to\infty}\frac{\lambda(\a^{[p^e]}/J^{[p^e]})}{p^{ed}}\hbox{ because }J\subseteq \a\\
&\geq\lim\limits_{e\to\infty}\frac{\lambda((J^{p^e(d-1)+1}+ J^{[p^e]})/J^{[p^e]})}{p^{ed}}\\
&=\lim\limits_{e\to\infty}\frac{\lambda(R/J)|\cA_e|}{p^{ed}}\\
&= 
 \lambda(R/J)\lim\limits_{e\to\infty}\frac{ \binom{p^e-1}{d}}{p^{ed}}\\
&= 
 e(R) \lim\limits_{e\to\infty}\frac{ \binom{p^e-1}{d}}{p^{ed}} \hbox{ because }\overline{J} = \m \\
&= \e(R)\frac{1}{d!}\\
\end{align*}
\end{proof}

\begin{example} Let $K$ be a perfect field of prime characteristic $p$, $n \geq 2$ be an integer, and $S=K[x_{ij} \mid 1 \leq i,j \leq n]$. Consider the ideal $I_2(X) \subseteq X$ generated by the $2 \times 2$ minors of the matrix $X = (x_{ij})_{1 \leq i,j \leq n}$. The ring $R=S/I_2(X)$ is Gorenstein of dimension $d=2n-1$, and the $a$-invariant $a:=a_d(R)$ is equal to $-n$. Let $\m$ be the irrelevant maximal ideal of $R$. Since $R$ is Gorenstein, we have that $\fpt(\m) = -a=n$. The ring $R$ can be viewed as a Segre product $K[X_1,\ldots,X_n] \ \# \ K[Y_1,\ldots,Y_n]$ and then, by \cite[Example 6.2]{HWY}, we have that $c^\m(\m) = n$. Since $R$ is Gorenstein and $\fpt(\m) = c^\m(\m)$, Theorem \ref{main thm} shows that $R$ is strongly $F$-regular. Even though this was already known,  because $R$ is a free summand of a polynomial ring \cite[Theorem 3.1 (e)]{HoHuStrong}, Theorem \ref{main thm} gives an alternative proof.
\end{example}

\begin{remark}
We think that the inequality proved in Theorem \ref{main thm} may not provide very meaningful bounds for the Hilbert-Samuel multiplicity. In fact, there are known bounds for multiplicities for $F$-pure and $F$-rational rings  \cite{HW}, which are better in several examples. However, Theorem \ref{main thm} is helpful to find lower bounds for the $F$-signature of rings as shown in Example \ref{exVraciu}.
\end{remark}

Using recent results of Singh, Takagi, and Varbaro \cite{STV}, we can extend Theorem \ref{main thm} to a more general setting. For a standard graded $F$-pure normal ring $(R,\m,K)$, let $X=\Spec(R)$, and let $K_X$ be the canonical divisor on $X$. The {\it anti-canonical cover of $X$} is defined as $\cR=\bigoplus_{n \geq 0} \mathcal{O}_X(-n K_X)$, and it is a very important object of study. It is known to be Noetherian in certain cases, which include the class of $\QQ$-Gorenstein rings, semigroup rings, and determinantal rings. Motivated by recent results on the minimal model program \cite{MMP1,MMP2,HC}, it is expected that $\cR$ is Noetherian when $R$ is strongly $F$-regular (see also \cite[Theorem 92]{Kol10}).

\begin{corollary} \label{coroll anticancover} Let $(R,\m,K)$ be a $d$-dimensional normal standard graded $K$-algebra that is Cohen-Macaulay, $F$-finite and $F$-pure. Assume that the anti-canonical cover of $R$ is Noetherian. If $\fpt(\m)=c^\m(\m),$ then $R$ is Gorenstein and strongly $F$-regular. Furthermore, $s(R) \geq \frac{\e(R)}{d!}$.
\end{corollary}
\begin{proof}
The equality $\fpt(\m) = c^\m(\m)$ forces $\fpt(\m) = -a_d(R)$ \cite[Theorem B]{DSNBFpurity}, and the latter implies that $R$ is quasi-Gorenstein \cite[Theorem A]{STV}. Since $R$ is assumed to be Cohen-Macaulay, $R$ is Gorenstein. The rest of the Corollary now follows from Theorem \ref{main thm}.
\end{proof}

We present a class of standard graded Gorenstein rings $(R,\m,K)$ that satisfy the equality $\fpt(\m)=c^\m(\m).$ In addition, we use the Theorem \ref{main thm} to find a lower bound for the $F$-signature. This example is possible thanks to recent computations of top socle degrees for diagonal hypersurfaces \cite{VraciuSocleDeg}.
\begin{remark}\label{Vraciu}
Let $S=K[x_1,\ldots,x_n]$ be a polynomial over an $F$-finite field of positive characteristic. Let $f=x^b_1+\ldots+x^b_n$, where $b\in\NN$, and set $R=S/fS$, with maximal ideal $\m$. Suppose that $\min\{p,n\}>b$, and set $\kappa=\left\lfloor \frac{p}{b}\right\rfloor.$
If 
$\left\lceil
\frac{n\kappa-n}{2}
\right\rceil \frac{1}{p}
\geq 1,$
then $c^{\m}(\m)=n-b$ \cite[Theorem 4.2]{VraciuSocleDeg}.
\end{remark}
\begin{example} \label{exVraciu}
Let $S=K[x_1,\ldots,x_n]$ be a polynomial over an $F$-finite field of positive characteristic, and $\m=(x_1,\ldots,x_n)$. Let $f=x^b_1+\ldots+x^b_n$, where $b\in\NN$, and let $R=S/fS$. Suppose that $p\equiv 1$ mod $b$,  $p\geq 2(b+1)$, and $n\geq 4b.$
Let  $\kappa=\left\lfloor \frac{p}{b}\right\rfloor$, and note that $\kappa=\frac{p-1}{b}$.
We have that
$$
\left\lceil
\frac{n\kappa-n}{2}
\right\rceil \frac{1}{p}
\geq 
\frac{n\kappa-n}{2} \cdot \frac{1}{p}
=\frac{n}{2} \cdot \frac{\kappa-1}{p}
=\frac{n}{2} \cdot \frac{p-b-1}{pb} = \frac{n}{2} \cdot \left(\frac{1}{b} - \frac{b+1}{pb}\right)
\geq \frac{n}{2} \cdot \frac{1}{2b}
\geq 1.
$$

It follows that $c^\m(\m)=n-b$, by Remark \ref{Vraciu}.  In addition, the top $a$-invariant of $R$ is $a_{n-1}(R)=b-n$. In order to show that $\fpt(\m)=n-b,$ we only need to prove that $R$ is an $F$-pure ring \cite[Theorem B]{DSNBFpurity}, because $R$ is Gorenstein. By Fedder's Criterion, Theorem \ref{Fedder}, it suffices to show that $f^{p-1}\not\in\m^{[p]}$. We note that $x^{b\kappa}\cdots x^{b\kappa}_n=(x_1\cdots x_n)^{p-1}\not\in \m^{[p]}.$ Since $n\kappa \geq b\kappa =p-1,$ there exists $\gamma_1,\ldots,\gamma_n$ such that $0\leq \gamma_i\leq \kappa$ and $\gamma_1+\ldots+\gamma_n=p-1$. Therefore, $x^{b\gamma_1}\cdots x^{b\gamma_n}_n$ is a monomial appearing with nonzero coefficient in $f^{p-1}$, and it does not belong to $\m^{[p]}.$ Hence, $R$ is $F$-pure, and $\fpt(\m)=c^\m(\m)$. In addition, since $e(R) = b$, we conclude that $\frac{b}{(n-1)!}\leq s(R)$ by Theorem \ref{main thm}.
\end{example}

We conclude this article with  one question motivated by Theorem \ref{Main GorGraded2}. 
\begin{question} \label{quest1} 
let $(R,\m,K)$ be a standard graded $K$-algebra that is $F$-finite and $F$-pure.
Does the equality $\fpt(\m)=c^\m(\m)$ imply that $R$ is $F$-rational?
\end{question}

\section*{Acknowledgments} We thank Craig Huneke for helpful discussions. We also thank Daniel Hern{\'a}ndez for inspiring conversations, in particular regarding Theorem \ref{Main 1}. We thank the anonymous referee for suggesting several improvements to the contents and the exposition of this article.

\bibliographystyle{alpha}
\bibliography{References}

\end{document}